\documentclass[11pt, twoside]{amsart}
\usepackage{amsmath,mathtools}
\usepackage{amsmath, amsthm, amssymb, amsfonts, enumerate}
\usepackage[colorlinks=true,linkcolor=blue,urlcolor=blue]{hyperref}
\usepackage{dsfont}
\usepackage{color}
\usepackage{geometry}
\usepackage{todonotes}
\usepackage{epstopdf}
\usepackage{bbm}
\usepackage{geometry}
\usepackage[utf8]{inputenc}
\usepackage{bm}
\usepackage{amsfonts}
\usepackage{amsfonts}
\usepackage{textcomp}
\usepackage{amssymb}
\usepackage{float}
\usepackage{tikz}
\usepackage{epsfig}
\usepackage{amsmath}
\usepackage[english]{babel}
\usepackage{a4}

\usepackage{enumerate}
\usepackage{tcolorbox}
\usepackage{soul}
\newcommand{\stkout}[1]{\ifmmode\text{\sout{\ensuremath{#1}}}\else\sout{#1}\fi}

\newtheorem{theorem}{Theorem}[section]

\newtheorem{remark}[theorem]{Remark}

\newtheorem{definition}[theorem]{Definition}

\usepackage{imakeidx}

\definecolor{red}{rgb}{1.0,0.0,0.0}

\definecolor{blu}{rgb}{0.0,0.0,1.0}

\definecolor{gre}{rgb}{0.03,0.50,0.03}

\usepackage[T1]{fontenc}

\title[Duality]{Duality methods in stochastic optimal control}

\date{\today}
\author[P. Bank]{Peter Bank}
\address{Peter Bank, Institut f\"{u}r Mathematik, Technische Universit\"{a}t Berlin, Berlin, Germany}
\email{\href{bank@math.tu-berlin.de}{bank@math.tu-berlin.de}}
\author[F. de Feo]{Filippo de Feo}
\address{Filippo de Feo, Institut f\"{u}r Mathematik, Technische Universit\"{a}t Berlin, Berlin, Germany}
\email{\href{defeo@math.tu-berlin.de}{defeo@math.tu-berlin.de}}

\numberwithin{equation}{section}

\begin{document}

\begin{abstract}
We prove two duality descriptions of the value function for a generic stochastic optimal control problem. These yield sharp bounds for the value function and also hold when the diffusion is controlled, a case left open by the literature so far. We also give duality descriptions for singular control problems.
\end{abstract}

\maketitle

\section{Introduction}
Optimal control problems of stochastic differential equations (SDEs)  have been extensively studied in the literature  using different approaches to characterize the value function, see e.g. the classical references \cite{fleming_soner,nisio,pham2009,touzi2012,yong_zhou}. 

Duality theory provides sharp bounds for the value function. Pathwise duality approaches have been most successfully applied in optimal stopping so far~\cite{belomestny_hildebrand_schoenmakers,davis-karatzas,rogers2002}, {providing a fundamental tool for the development of numerical schemes.} For optimal control, there is early work by \cite{burstein_davis} and, more recently, using rough paths techniques in \cite{diehl_friz_gassiat}. In these approaches, one  characterizes the value functions in terms of suitable pathwise optimal control problems involving penalizations.  Observing that the pathwise problem may degenerate when the diffusion is controlled, these accounts refrained from considering this possibility further; see also  \cite{allan_cohen}. 
However, we observe that  in stochastic optimal control problems in discrete time, \cite{rogers2007} constructs suitable penalty terms without such hard restrictions on the system dynamics.  {In fact, \cite{henry-litterer-ren} allows the control to enter the diffusion coefficient and their dual scheme recovers the value  as the vanishing-mesh limit of essentially discrete-time problems, leaving open a direct dual representation for the original admissible control class.}

\paragraph{\textbf{Contributions.}}
(i) Starting from the above observations, we are led to investigate the duality approach for optimal control problems of SDEs with controlled diffusion, with the goal of settling the question: 
$$\textit{Is a duality representation possible when the diffusion is controlled?}$$
 We  show that, indeed, a duality representation holds, both when the value function is a classical solution of  the Hamilton-Jacobi-Bellman (HJB) equation (Theorem \ref{th:duality_C12}) and when this smoothness property is not assumed (Theorem \ref{th:duality_viscosity}). In fact, {in either case,} we provide two duality descriptions: the first involves a family of pathwise stochastic optimization problems that may degenerate when the diffusion part is controlled, rendering their analysis rather delicate. In that case, however, the second description may be more convenient as it only involves a pointwise optimization. The main advantage of these descriptions is that the dual formulations provide sharp upper bounds for the value function, which can be used to build numerical schemes \cite{lauriere-yang}. Our results and proofs are very simple and direct: we do not rely on rough paths techniques as in  \cite{diehl_friz_gassiat} or on approximations arguments as in \cite{henry-litterer-ren}.

 (ii) 
{The method we propose is very robust. Indeed, in Section \ref{sec:singular}, we provide an extension including singular control problems, where we prove duality following very similar lines.    {While we address finite time horizon problems here, the result can be easily extended to the infinite-horizon case.}
Furthermore, the results of Section \ref{sec:setting} have already been extended to stochastic control with non-Makovian random coefficients in \cite{BDDF}.} 
\paragraph{\textbf{Outline.}}
The manuscript is organized as follows. In Section \ref{sec:setting} we provide a duality theory for stochastic optimal control. In Section \ref{sec:singular}, we provide an extensions of the duality to singular control problems. 

\section{Duality for stochastic optimal control}\label{sec:setting}
Let  $(\Omega, \mathcal{F}, \mathbb{P})$ be a complete probability space, 
 $\{W_s\}_{s \geq 0}$  an $m$-dimensional standard Brownian motion on $(\Omega, \mathcal{F}, \mathbb{P})$, and let $\mathcal{F}_s^W$ be the augmented Brownian filtration. Given  a nonempty family  of  random variables $\{X_\pi\}$ defined on $(\Omega, \mathcal{F}, \mathbb{P})$, the essential supremum of $X_\pi$, denoted by $\operatorname*{esssup}_{\pi}X_\pi$, is a random variable $X^*$ satisfying: for all $\pi$, $X^* \geq X_\pi,$ a.s.; if $Y$ is a random variable such that $Y \geq X_\pi$ for all $\pi$, then $Y \geq X^*$, a.s.

On this setup, consider a controlled SDE of the form
\begin{equation}\label{eq:state_eq}
    dX_s=b(s,X_s,\pi_s)ds+\sigma (s,X_s,\pi_s)dW_s, \quad X_t=x\in \mathbb R^d,
\end{equation}
where 
$\pi\in \mathcal{U}:= \{\pi \colon \Omega \times [0,T]\to U \subset \mathbb R^h , \ \mathcal{F}_s^W-\textit{progressively measurable} \}$ is an admissible control, $b:[0,T]\times \mathbb R^d\times U\to \mathbb R^d,$ and $ \sigma :[0,T]\times \mathbb R^d\times U\to \mathbb R^{d\times m}$ are measurable.  Throughout the manuscript, we assume that the SDE  \eqref{eq:state_eq} has a unique strong solution $ X_s^{t,x,\pi}$, for each $t\in [0,T],x\in \mathbb R^d, \pi \in \mathcal U$, i.e.  
\begin{equation}
    X_s^{t,x,\pi}=x+\int_t^s b(r,X_r^{t,x,\pi},\pi_r)dr+\int_t^s \sigma (r,X_r^{t,x,\pi},\pi_r)dW_r.
\end{equation}

The goal is to maximize, over all admissible controls $\pi \in \mathcal U,$ a cost functional of the form
$$J(t,x,\pi):=\mathbb E\left[\int_t^T l(s,X_s^{t,x,\pi},\pi_s) ds+g(X_T^{t,x,\pi}) \right].$$

Following the dynamic programming approach, we define the value function
$$V:[0, T] \times \mathbb R^d \to \mathbb R,\quad V(t,x):=\sup_{\pi \in \mathcal U}\mathbb E\left[\int_t^T l(s,X_s^{t,x,\pi},\pi_s) ds+g(X_T^{t,x,\pi}) \right],$$
where measurable  $l:[0,T]\times \mathbb R^d\times U\to \mathbb R, $ $g :\mathbb R^d\to \mathbb R$ are the running reward and final reward, respectively. Throughout the manuscript, will assume that $J$ and $V$ are well-defined.

The HJB equation associated to the problem  is
\begin{equation}\label{eq:HJB}
    -\partial_t v-H(t,x,\partial_{x} v,\partial^2_{x^2} v)=0, \quad  \forall(t, x) \in(0, T) \times \mathbb R^d, \quad v(T, x) =g(x), \quad  \forall x \in \mathbb R^d ,
\end{equation}
where the current value Hamiltonian and the Hamiltonian are defined by
\begin{align*}
    &H^{cv}(t,x,p,Z,\pi)=b(t,x,\pi)  p+ \frac 1 2 \operatorname{Tr}[\sigma\sigma^T(t,x,\pi) Z] + l(t,x,\pi) ,\\
    &H(t,x,p,Z)=\sup _{\pi \in U}H^{cv}(t,x,p,Z,\pi)
\end{align*}
and are assumed to be continuous functions.
\subsection{Classical solutions}\label{sec_smooth}
In this subsection, we assume that the value function $V \in C^{1,2}((0, T) \times \mathbb{R}^d) \cap C((0, T] \times \mathbb{R}^d)$ is a classical solution of HJB \eqref{eq:HJB} and we prove a duality theorem. In preparation of also the non-smooth case, let us introduce the notion of sub- and supersolution.
\begin{definition}\label{def:classcal_sol}
    A function \( v \in C^{1,2}((0, T) \times \mathbb{R}^d) \cap C([0, T] \times \mathbb{R}^d) \) is said to be a classical subsolution (resp. supersolution, resp. solution) if $v(T, x) \leq g(x)$ (resp. $v(T, x) \geq g(x)$, resp. $v(T, x) = g(x)$) for all  $x \in \mathbb{R}^d$ and
    \[
    -\partial_t v(t,x) - H\left(t, x, D v(t,x), \partial^2_{x^2} v(t,x)\right) \leq 0 ,\quad (\textit{resp. }\geq 0, \textit{ resp. }= 0) \quad  \forall (t,x) \in (0, T) \times \mathbb{R}^d.
    \]
    We denote the sets of classical sub/supersolutions, respectively by $\mathcal S^-$, $\mathcal S^+$.
\end{definition}

For  $h \in C^{1,2}( (0,T)\times \mathbb R^d)$, we will denote 
\begin{align*}
M_{t,T}^{t,x,\pi,h} &:=h(T,X_T^{t,x,\pi})-h(t,x)-\int_t^T\big[\partial_s h(s,X_s^{t,x,\pi})+b(s,X_s^{t,x,\pi},\pi_s)  \partial_x h(s,X_s^{t,x,\pi})\\
&\quad + \frac 1 2\operatorname{Tr}(\sigma\sigma^T(s,X_s^{t,x,\pi},\pi_s) \partial^2_{x^2} h(s,X_s^{t,x,\pi}))\big]  ds=\int_t^T \sigma(s,X_s^{t,x,\pi},\pi_s)\partial_x h(s,X_s^{t,x,\pi}) dW_s,
\end{align*}
where the last equality follows by Ito's formula. 
\begin{theorem}\label{th:duality_C12}Assume that $V \in C^{1,2}((0, T) \times \mathbb{R}^d) \cap C((0, T] \times \mathbb{R}^d)$ is a classical solution of HJB and {$\mathbb E [M_{t,T}^{t,x,\pi,V}]\leq 0$, for all $\pi\in \mathcal U$.} Then 
\begin{align*}
V(t,x)&=\min_{h \in \mathcal H} V_1^h(t,x)=\min_{h \in \mathcal H} V_2^h(t,x),
\end{align*}
where $\mathcal H:=\left\{h \in C^{1,2}( (0,T)\times \mathbb R^d): h(T,x)=g(x), \mathbb E [M_{t,T}^{t,x,\pi,h}]\leq 0,\forall \pi\in \mathcal U\right \}$ and
\begin{align*}
& V_1^h(t,x):=\mathbb E\Big[\operatorname*{esssup}_{\pi \in {\mathcal U}}\left ( \int_t^T l(s,X_s^{t,x,\pi},\pi_s) ds+g(X_T^{t,x,\pi})  - M_{t,T}^{t,x,\pi,h} \right)   \Big]
\\
&\quad \quad \quad \quad \equiv h(t,x) + \mathbb E\Big[\operatorname*{esssup}_{\pi \in {\mathcal U}} \int_t^T\partial_s h(s,X_s^{t,x,\pi}) +H^{cv}(s,X_s^{t,x,\pi},\partial_{x} h(s,X_s^{t,x,\pi}) ,\partial^2_{x^2} h(s,X_s^{t,x,\pi}),\pi_s)    ds  \Big],\\
&V_2^h(t,x):=h(t,x) + \int_t^T \sup_{y\in \mathbb R^d}  \left[  \partial_s h(s,y) +  H(s,y,\partial_x h(s,y),\partial^2_{x^2} h(s,y)) \right] ds.
\end{align*}
\end{theorem}
The strength of this theorem is that it does not merely give an abstract reformulation of the value function, but produces sharp dual bounds for it, while the choice $h=V$ makes the bound exact, so there is no duality gap. Of course, the choice $h=V$ is only possible when the value function is smooth in the sense that $V \in \mathcal H$. However,  in the next subsection, we will also provide a duality result  under non-smoothness of $V$. 

The condition $\mathbb E[M_{t,T}^{t,x,\pi,h}]\leq 0$, for any $\pi$, is met, for instance, under polynomial growth of coefficients and corresponding moments of the state.
\begin{proof}
Let $h \in \mathcal H$ and use Ito's formula and the condition $\mathbb E [M_{t,T}^{t,x,\pi,h}]\leq 0$, to have
\begin{align*}
V(t,x)&=\sup_{\pi \in {\mathcal U}} \mathbb E\left[\int_t^T l(s,X_s^{t,x,\pi},\pi_s) ds+h(T,X_T^{t,x,\pi}) \right]\\
&\leq h(t,x)+ \sup_{\pi \in {\mathcal U}}\mathbb E\left[  \int_t^T\partial_s h(s,X_s^{t,x,\pi}) + H^{cv}(s,X_s^{t,x,\pi},\partial_{x} h(s,X_s^{t,x,\pi}) ,\partial^2_{x^2} h(s,X_s^{t,x,\pi}),\pi_s )   ds\right]\\
&\leq V_1^h(t,x)\leq V_2^h(t,x) .
\end{align*}
Taking $h=V\in \mathcal H$ and using the HJB equation \eqref{eq:HJB}, we have
$$ V(t,x)=V_1^V(t,x)=V_2^V(t,x).$$
The claim follows.
\end{proof}
\begin{remark}\label{rem:sigma_depdendent_independent_control}
    Observe that when $\sigma=\sigma(s,x,u)$ depends on $u$, the dual stochastic problem might degenerate. Indeed,  one might have (see \cite[Remark 13]{diehl_friz_gassiat})
    \begin{align*}
&\operatorname*{esssup}_{\pi \in {\mathcal U}}{\int_t^T} \left[   \partial_s h(s,X_s^{t,x,\pi}) +H^{cv}(X_s^{t,x,\pi},\partial_{x} h(s,X_s^{t,x,\pi}) ,\partial^2_{x^2} h(s,X_s^{t,x,\pi}),\pi_s)\right]ds \\
&=\int_t^T  \sup_{y\in \mathbb R^d}\left[  \partial_s h(s,y) +  H(s,y,\partial_x h(s,y),\partial^2_{x^2} h(s,y)) \right]ds.
\end{align*}
In that case, the $V_1^h$ and $V_2^h$ obviously coincide. In general, however, it may be that $V_1^h$ gives a tighter bound than $V_2^h$, albeit at the cost of having to solve a pathwise optimal control problem as opposed to a pointwise optimization.  
\end{remark}
\subsection{Non-smooth case}\label{sec_visco}
In this subsection, we  drop the assumption that $V\in C^{1,2}$, and we prove the following duality theorem. 

{We will use the following weaker assumption.} We recall that the set of classical supersolutions $\mathcal S^+$ was introduced in Definition \ref{def:classcal_sol}.
\begin{theorem}\label{th:duality_viscosity}
Let \(\tilde{\mathcal S}^+
:=S^+\cap\tilde{\mathcal H}\neq \emptyset\) and $t\in [0,T],x\in \mathbb R^d$ such that
\begin{equation}\label{hp:Splus}
     \inf_{h \in \tilde{\mathcal S}^+} h(t,x)\leq V(t,x), 
    \end{equation}  
    where 
    $\tilde{\mathcal H}:= \left\{h \in C^{1,2}( (0,T)\times \mathbb R^d): \mathbb E [M_{t,T}^{t,x,\pi,h}]\leq 0 , \forall \pi \in \mathcal U\right \}$.
    Then:
\begin{align*}
V(t,x)&=\inf_{h \in \tilde{ \mathcal H}} \tilde{V}_1^h(t,x)=\inf_{h \in \tilde{ \mathcal H}} \tilde{V}_2^h(t,x),
\end{align*}
where 
\begin{align*}
& \tilde{V}_1^h(t,x):=\mathbb E\Big[\operatorname*{esssup}_{\pi \in {\mathcal U}}\left ( \int_t^T l(s,X_s^{t,x,\pi},\pi_s) ds+g(X_T^{t,x,\pi})  - M_{t,T}^{t,x,\pi,h} \right)   \Big]\\
&  \equiv h(t,x) + \mathbb E\bigg[\operatorname*{esssup}_{\pi \in {\mathcal U}} \bigg ( \int_t^T\partial_s h(s,X_s^{t,x,\pi}) +H^{cv}(X_s^{t,x,\pi},\partial_{x} h(s,X_s^{t,x,\pi}) ,\partial^2_{x^2} h(s,X_s^{t,x,\pi}),\pi_s)    ds\\&
\quad \quad \quad \quad \quad \quad \quad \quad \quad \quad \quad \quad\quad + g(X_T^{t,x,\pi})- h(T,X_T^{t,x,\pi}) \bigg )  \bigg],\\
&\tilde{V}_2^h(t,x):=h(t,x) + \int_t^T \sup_{y\in \mathbb R^d}  \left[  \partial_s h(s,y) +  H(s,y,\partial_x h(s,y),\partial^2_{x^2} h(s,y)) \right] ds + \sup_{y\in \mathbb R^d} \left [g(y)-h(T,y) \right].
\end{align*}
\end{theorem}
\begin{remark}
Condition \eqref{hp:Splus}  is 
proved,  under suitable regularity of the coefficients, in \cite{krylov2000}, \cite{barles_jakobsen},  \cite[Proof of Theorem 16]{diehl_friz_gassiat}, by constructing a sequence of smooth supersolutions $V_\epsilon$ converging to $V$, i.e.  
\begin{equation*}
     \inf_{h \in \tilde {\mathcal S}^+} h(t,x)\leq V_\epsilon(t,x)\xrightarrow{\epsilon \to 0} V(t,x),   
    \end{equation*}
\end{remark}
\begin{proof}
Let $h\in \tilde{\mathcal H}$. Again, by   Ito's formula and  the condition $\mathbb E [M_{t,T}^{t,x,\pi,h}]\leq 0$, we have
\begin{align*}
V(t,x)&=\sup_{\pi \in {\mathcal U}} \mathbb E\left[\int_t^T l(s,X_s^{t,x,\pi},\pi_s) ds+h(T,X_T^{t,x,\pi}) + g(X_T^{t,x,\pi})- h(T,X_T^{t,x,\pi}) \right]\\
&\leq h(t,x)+ \sup_{\pi \in {\mathcal U}}\mathbb E\bigg[  \int_t^T\partial_s h(s,X_s^{t,x,\pi}) + H^{cv}(s,X_s^{t,x,\pi},\partial_{x} h(s,X_s^{t,x,\pi}) ,\partial^2_{x^2} h(s,X_s^{t,x,\pi}),\pi_s )   ds\\
&\quad \quad \quad\quad \quad \quad\quad \quad  + g(X_T^{t,x,\pi})- h(T,X_T^{t,x,\pi})  \bigg]\\
&\leq  \tilde{V}_1^h(t,x)\leq  \tilde{V}_2^h(t,x).
\end{align*}
Taking the infimum over $h\in \tilde{\mathcal H}$, recalling that $ \tilde{\mathcal S}^+\subset \tilde{ \mathcal H}$ is the set of classical supersolution of the HJB equation, and using Assumption \ref{hp:Splus}, we have
\begin{align*}
V(t,x)&\leq \inf_{h \in \tilde{ \mathcal H}} \tilde{V}_1^h(t,x)\leq \inf_{h \in \tilde{ \mathcal H}} \tilde{V}_2^h(t,x) \leq  \inf_{h  \in \tilde{\mathcal S}^+} h(t,x)\leq V(t,x) .
\end{align*}
The claim follows.
\end{proof}
Clearly, Remark \ref{rem:sigma_depdendent_independent_control} applies also here.

\section{Duality for singular control problems}\label{sec:singular}
In this section, we extend the problem of Section \ref{sec:setting} to singular controls in the sense of \cite[Chapter VIII]{fleming_soner}.  Let $K\subset\mathbb R^d$ be a compact set of admissible unit directions and let $c:K\to[0,\infty)$ be continuous. An admissible singular control is a pair $\alpha=(\xi,\kappa)\in\mathcal A$, where $\xi$ is an adapted c\`adl\`ag nondecreasing process with $\xi_{t-}=0$, and $\kappa$ is progressively measurable, $K$-valued $d\xi$-a.e. The cost associated with $\alpha \in \mathcal A$ is $\int_{[t,T]}c(\kappa_s)d\xi_s$. Then, the state equation is
\begin{equation}\label{eq:state_eq_singular}
\begin{split}
X_s^{t,x,\pi,\alpha}={}&x+\int_t^s b(r,X_{r-}^{t,x,\pi,\alpha},\pi_r)dr
+\int_t^s\sigma(r,X_{r-}^{t,x,\pi,\alpha},\pi_r)dW_r+\int_{[t,s]}\kappa_r\,d\xi_r.
\end{split}
\end{equation}
Thus $\Delta X_s^{t,x,\pi,\alpha}=\kappa_s\Delta\xi_s$. The corresponding reward and value function are
\begin{align*}
J(t,x,\pi,\alpha)
&:=\mathbb E\left[\int_t^T l(s,X_s^{t,x,\pi,\alpha},\pi_s)ds
-\int_{[t,T]}c(\kappa_s)d\xi_s+g(X_T^{t,x,\pi,\alpha})\right],\\
V(t,x)&:=\sup_{(\pi,\alpha)\in\mathcal U\times\mathcal A}
J(t,x,\pi,\alpha).
\end{align*}
Since interventions at $T$ are allowed, we impose the standard terminal compatibility condition
\[
g(x+r\kappa)-g(x)\leq r c(\kappa),
\qquad x\in\mathbb R^d,\quad r\geq0,\quad\kappa\in K.
\]
For differentiable $g$, this follows from $G(\partial_xg)\leq0$ below; without this condition the terminal function must be replaced by its singular-control face-lift. We assume throughout this section the usual well-posedness and integrability conditions.

In addition to the Hamiltonian $H$ of Section \ref{sec:setting}, define the singular Hamiltonian
\[
G(p):=\sup_{\kappa\in K}\{p\cdot\kappa-c(\kappa)\}.
\]
The HJB equation is now the variational inequality
\begin{equation}\label{eq:HJB_singular}
\max\left\{\partial_t v+H(t,x,\partial_xv,\partial^2_{x^2}v),
G(\partial_xv)\right\}=0,
\qquad v(T,x)=g(x).
\end{equation}

In this section, we assume that the value function $V\in C^{1,2}([0,T]\times\mathbb R^d)$ is a classical solution of \eqref{eq:HJB_singular} and we prove a duality theorem.

For $h\in C^{1,2}([0,T]\times\mathbb R^d)$, let $\xi^c$ denote the continuous part of $\xi$ and set
\begin{align}
M_{t,T}^{t,x,\pi,\alpha,h}:={}&h(T,X_T^{t,x,\pi,\alpha})-h(t,x)\notag\\
&-\int_t^T\Bigg[\partial_sh(s,X_{s-}^{t,x,\pi,\alpha})
+b(s,X_{s-}^{t,x,\pi,\alpha},\pi_s)\partial_xh(s,X_{s-}^{t,x,\pi,\alpha})\notag\\
&\hspace{3.8cm}+\frac12\operatorname{Tr}\left[\sigma\sigma^T(s,X_{s-}^{t,x,\pi,\alpha},\pi_s)
\partial^2_{x^2}h(s,X_{s-}^{t,x,\pi,\alpha})\right]\Bigg]ds\notag\\
&-\int_{[t,T]}\partial_xh(s,X_{s-}^{t,x,\pi,\alpha})\cdot\kappa_s\,d\xi_s^c
-\sum_{t\leq s\leq T}\left[h(s,X_s^{t,x,\pi,\alpha})-h(s,X_{s-}^{t,x,\pi,\alpha})\right]\notag\\
={}&\int_t^T\partial_xh(s,X_{s-}^{t,x,\pi,\alpha})
\sigma(s,X_{s-}^{t,x,\pi,\alpha},\pi_s)dW_s .
\label{eq:martingale_singular}
\end{align}
where the last equality follows by the It\^o formula.
\begin{theorem}[Smooth singular-control duality]\label{th:duality_singular}
Assume that $V\in C^{1,2}([0,T]\times\mathbb R^d)$ is a classical solution of \eqref{eq:HJB_singular} and $\mathbb E [M_{t,T}^{t,x,\pi,\alpha,V}]\leq 0$, for all $\pi \in \mathcal U, \alpha \in \mathcal A$. Then
\begin{align*}
V(t,x)&=\min_{h\in\mathcal H}V_{1}^h(t,x)
=\min_{h\in\mathcal H}V_{2}^h(t,x),
\end{align*}
where $\mathcal H:=\{h\in C^{1,2}([0,T]\times\mathbb R^d):
h(T,x)=g(x), \ \mathbb E [M_{t,T}^{t,x,\pi,\alpha,h}]\leq 0, \forall \pi\in \mathcal U,\alpha \in \mathcal A,  \ G(\partial_xh)\leq0\}$ and
\begin{align*}
&V_{1}^h(t,x):=\mathbb E\Bigg[\operatorname*{esssup}_{(\pi,\alpha)\in\mathcal U\times\mathcal A}\Bigg(
\int_t^T l(s,X_s^{t,x,\pi,\alpha},\pi_s)ds-\int_{[t,T]}c(\kappa_s)d\xi_s+g(X_T^{t,x,\pi,\alpha})-M_{t,T}^{t,x,\pi,\alpha,h}\Bigg)\Bigg]\\
&\quad\equiv h(t,x)+\mathbb E\Bigg[\operatorname*{esssup}_{(\pi,\alpha)\in\mathcal U\times\mathcal A}\Bigg(
\int_t^T\Big[\partial_sh(s,X_{s-}^{t,x,\pi,\alpha})\\
&\quad\quad\quad\quad\quad \quad \quad\quad+H^{cv}\big(s,X_{s-}^{t,x,\pi,\alpha},
\partial_xh(s,X_{s-}^{t,x,\pi,\alpha}),
\partial^2_{x^2}h(s,X_{s-}^{t,x,\pi,\alpha}),\pi_s\big)\Big]ds\\
&\quad\quad\quad\quad\quad \quad \quad\quad+\int_t^T
\Big[
\partial_xh(s,X_{s-}^{t,x,\pi,\alpha})\cdot\kappa_s
-c(\kappa_s)
\Big]\,d\xi_s^c\\
&\quad\quad\quad\quad\quad\quad \quad\quad
+\sum_{t\leq s\leq T}
\Big[
h(s,X_s^{t,x,\pi,\alpha})
-h(s,X_{s-}^{t,x,\pi,\alpha})
-c(\kappa_s)\Delta\xi_s
\Big]
\Bigg)\Bigg],\\
&V_{2}^h(t,x):=h(t,x)+\int_t^T\sup_{y\in\mathbb R^d}
\left[\partial_sh(s,y)+H(s,y,\partial_xh(s,y),\partial^2_{x^2}h(s,y))\right]ds.
\end{align*}
\end{theorem}

\begin{proof}
Let $h\in\mathcal H$ and use the It\^o formula and the condition $\mathbb E [M_{t,T}^{t,x,\pi,h}]\leq 0$:
\begin{align*}
V(t,x)
&=\sup_{(\pi,\alpha)\in\mathcal U\times\mathcal A}
\mathbb E\left[
\int_t^T l(s,X_s^{t,x,\pi,\alpha},\pi_s)\,ds
-\int_{[t,T]}c(\kappa_s)\,d\xi_s
+h(T,X_T^{t,x,\pi,\alpha})
\right]\\
&\leq h(t,x)+
\sup_{(\pi,\alpha)\in\mathcal U\times\mathcal A}
\mathbb E\Bigg[
\int_t^T
\Big[
\partial_s h(s,X_{s-}^{t,x,\pi,\alpha})\\
&\hspace{2.5cm}
+H^{cv}\Big(
s,X_{s-}^{t,x,\pi,\alpha},
\partial_xh(s,X_{s-}^{t,x,\pi,\alpha}),
\partial^2_{x^2}h(s,X_{s-}^{t,x,\pi,\alpha}),
\pi_s
\Big)
\Big]\,ds\\
&\hspace{1.5cm}
+\int_t^T
\Big[
\partial_xh(s,X_{s-}^{t,x,\pi,\alpha})\cdot\kappa_s
-c(\kappa_s)
\Big]\,d\xi_s^c\\
&\hspace{1.5cm}
+\sum_{t\leq s\leq T}
\Big[
h(s,X_s^{t,x,\pi,\alpha})
-h(s,X_{s-}^{t,x,\pi,\alpha})
-c(\kappa_s)\Delta\xi_s
\Big]
\Bigg]\\
&\leq V_1^h(t,x)
\leq V_2^h(t,x).
\end{align*}
where we have also used the constraint  $G(\partial_xh(s,x))\leq 0$ and the inequality
\begin{align*}
h(s,X_{s-}+r\kappa_s)-h(s,X_{s-})-r c(\kappa_s)
=\int_0^r\left[\partial_xh(s,X_{s-}+q\kappa_s)\cdot\kappa_s-c(\kappa_s)\right]dq\leq0,
\end{align*}
with $r=\Delta\xi_s.$

Next, we have   $V\in\mathcal H$; indeed, this follows from the HJB variational inequality \eqref{eq:HJB_singular}, which implies
$
G(\partial_xV(s,y))\leq0 $ and $\partial_sV(s,y)
+H(s,y,\partial_xV(s,y),\partial^2_{x^2}V(s,y))\leq0$
for every $(s,y)\in(0,T)\times\mathbb R^d$. Moreover, we have
\begin{align*}
V_{2}^V(t,x)
={}&V(t,x)+\int_t^T\sup_{y\in\mathbb R^d}
\Big[\partial_sV(s,y)+H(s,y,\partial_xV(s,y),
\partial^2_{x^2}V(s,y))\Big]ds\leq V(t,x).
\end{align*}
Combining this inequality with the preceding inequalities evaluated at $h=V$ gives
\[
V(t,x)\leq V_{1}^V(t,x)
\leq V_{2}^V(t,x)\leq V(t,x).
\]
The statement of the proposition follows.
\end{proof}
\begin{remark}[Nonsmooth extension]
A nonsmooth counterpart of Theorem~\ref{th:duality_singular} could be developed along the lines of Subsection~\ref{sec_visco}, provided that $V$ can be approximated from above by smooth supersolutions $h$ satisfying both
\[
\partial_t h+H(t,x,\partial_xh,\partial^2_{x^2}h)\leq0
\qquad\text{and}\qquad
G(\partial_xh)\leq0,
\]
together with $h(T,\cdot)\geq g$. We leave this this extension to future work because, under the present assumptions, the available viscosity theory does not provide precisely the required global smooth approximation while preserving the gradient constraint.
\end{remark}

\paragraph{\textbf{Acknowledgments.}}
The authors acknowledge funding by the Deutsche Forschungsgemeinschaft (DFG, German Research Foundation) – CRC/TRR 388 "Rough Analysis, Stochastic Dynamics and Related Fields" – Project ID 516748464.
\begin{small}
\bibliography{refs}
\bibliographystyle{plain}
\end{small}
\end{document}